\RequirePackage{ifpdf}
\ifpdf
 	\documentclass[pdftex]{amsart}
 	\else
 	\documentclass[dvips]{amsart}
\fi

\usepackage[T1]{fontenc}
\usepackage[utf8]{inputenc}
\usepackage{amsfonts,amsthm,latexsym,amsmath,amssymb,amscd,amsmath,mathrsfs,caption,subcaption, epsf}
\usepackage{graphicx}
\ifpdf
	\usepackage{epstopdf}
	\DeclareGraphicsExtensions{.png,.jpg,.eps,.epsf}
\fi

\newtheorem{theorem}{Theorem}
\newtheorem{proposition}[theorem]{Proposition}
\newtheorem{lemma}[theorem]{Lemma}
\newtheorem{notation}[theorem]{Notation}
\newtheorem{definition}[theorem]{Definition}

\newtheorem{example}[theorem]{Example}
\newtheorem{remark}[theorem]{Remark}
\newtheorem{conjecture}[theorem]{Conjecture}
\newtheorem{problem}[theorem]{Problem}

\newcommand{\defin}[1]{{\emph{#1}}}

\newcommand{\C}{\mathbb{C}}

\newcommand{\al}{\alpha}

\newcommand{\xvec}{\mathbf{x}}

\newcommand{\alphavec}{{\boldsymbol\alpha}}

\newcommand{\infmatA}{{\mathcal{A}}}
\newcommand{\infmatI}{{\mathcal{I}}}

\newcommand{\eigenlocus}{{\mathcal{E}}}
\newcommand{\polybasis}{{\mathfrak{B}}}
\newcommand{\M}{{ Mat}}

\DeclareMathOperator*{\rank}{rank}

\title{Around multivariate Schmidt-Spitzer theorem}


\author[P.~Alexandersson]{Per Alexandersson}
\address{Department of Mathematics,
   Stockholm University,
   S-10691, Stockholm, Sweden}
\email{per@math.su.se}

\author[B.~Shapiro]{Boris Shapiro}
\address{Department of Mathematics,
   Stockholm University,
   S-10691, Stockholm, Sweden}
\email{shapiro@math.su.se}
\begin{document}

\begin{abstract}
Given an arbitrary complex-valued infinite matrix $\infmatA=(a_{ij}),\;
i=1,\dotsc,\infty; j=1,\dotsc,\infty$  and a positive integer $n$ we introduce a
naturally associated  polynomial basis $\polybasis_\infmatA$ of
$\C[x_0,\dotsc,x_n]$.
We discuss some properties of the locus of  common zeros of all polynomials 
in $\polybasis_A$ having  a given degree $m$; the latter locus can be
interpreted as the spectrum of 
the $m\times (m+n)$-submatrix of $\infmatA$ formed by its  $m$ first rows and
$(m+n)$ first columns. 
We initiate the study of the asymptotics of these spectra when $m\to \infty$ in
the case when $\infmatA$ is a banded Toeplitz matrix.
In particular, we present and partially prove a conjectural multivariate analog
of the well-known Schmidt-Spitzer theorem which describes 
 the spectral asymptotics for the sequence of principal minors of an arbitrary
banded Toeplitz matrix.
Finally, we discuss relations between polynomial bases $\polybasis_\infmatA$ and
multivariate  orthogonal polynomials. 
\end{abstract}

\maketitle

\section{Introduction}

The approach of this paper is motivated by the modern interpretation of the 
Heine-Stieltjes multiparameter spectral problem as presented in \cite{Sh} and 
\cite{ShSh}. Let us recall some relevant results in the matrix set-up.  

Given integers $m>0$ and $n\ge 0$ consider the space $\M(m,m+n)$ of 
complex-valued $m \times (m+n)$-matrices. For $s=0,\dots, n$ 
define the $s$-th unit matrix 
$$\infmatI_s := (\delta_{s+i-j})\in \M(m,m+n).$$ 
(In what follows the sizes of matrices can be infinite.)

\begin{definition}[see \cite{ShSh}]
Given a matrix $A\in\M(m,m+n)$ define its \defin{eigenvalue locus}
$\eigenlocus_A$ as
\[
\eigenlocus_A := \left\{ (x_0,x_1,\dotsc,x_n) \in \C^{n+1} \colon \rank \left( A
- \sum_{s=0}^n x_s \infmatI_s \right) < m \right\}.
\]
For $n=0,$ $\eigenlocus_A$ coincides with the usual set of eigenvalues of a
square matrix $A$.
\end{definition}

\begin{proposition}[ Lemma~1 of \cite{ShSh}]\label{prop:eig}
For arbitrary $A\in \M(m,m+n)$ the eigenvalue locus $\eigenlocus_A$ consists of
$\binom{m+n}{n+1}$ points
counting multiplicities. In other words, counting multiplicities there exist 
$\binom{m+n}{n+1}$ eigenvalue tuples $(x_0,x_1,\dotsc,x_n)$  such that  
$A-\sum_{s=0}^{n}x_s \infmatI_s$ has rank smaller than $m$.
\end{proposition}

\begin{remark}
Notice that for $n>0,$ the locus  $\eigenlocus_A$  is not a complete 
intersection since it is given by the vanishing of all maximal minors of $A$. 
(A similar phenomenon can be observed for common zeros of 
multivariate Schur polynomials, since Schur polynomials are given by determinant
formulas.) 
\end{remark}

\begin{notation}\label{deff}
Given an infinite matrix $\infmatA=(a_{ij}),\; i=1,\dotsc,\infty;
j=1,\dotsc,\infty$, an integer $n\geq 0$, and an $m$-tuple of positive integers
$I=(i_1,i_2,\dots, i_m)$  satisfying $1\le i_1<i_2<\dotsc <i_m\leq m+n$,   
consider the submatrix $A_I$ of $\infmatA - \sum_{s=0}^n x_s \infmatI_s$ formed
by the first $m$ rows and $m$ columns indexed by $I.$
Define 
\begin{equation}\label{eq:main}
P^I_\infmatA(x_0,x_1,\dotsc,x_n):=\det A_I.
\end{equation}
\end{notation}
Evidently, $P^I_\infmatA(x_0,\dotsc,x_n)$ is a maximal minor of the principal $m
\times (m+n)$ submatrix  of $\infmatA- \sum_{s=0}^n x_s \infmatI_s$ formed by
its $m$ first rows and $m+n$ first columns. Therefore
$P^I_\infmatA(x_0,\dotsc,x_n)$ is  a polynomial in $x_0,\dotsc,x_n.$

\begin{proposition}\label{prop:basis}
In the above notation the
following holds:
\begin{enumerate}

\item[(i)] for any multiindex $I$ with 
$|I|=m$, $\deg P^{I}_\infmatA(x_0,\ldots, x_n)=m$; 

\noindent
\item[(ii)] all $\binom{m+n}{m}$ polynomials 
$P^I_\infmatA(x_0,...,x_n)\in \C[x_0,\dots,x_n]$ with $|I|=m$ are linearly 
independent which implies that the totality of all polynomials 
$P^I_\infmatA(x_0,...,x_n)$ is a linear basis of $\C[x_0,\dots,x_n]$; 

\noindent
\item[(iii)] the set $\eigenlocus_\infmatA^{(m)}$ of 
common zeros of all  $P^I_\infmatA(x_0,...,x_n)$ with $|I|=m$ 
is a finite subset of $\C^{n+1}$ of cardinality $\binom{m+n}{n+1}$ counting
multiplicities.  
\end{enumerate}
\end{proposition}

\begin{remark}\label{rem:2} Notice that for  $\binom {m+n}{m}$ randomly chosen  
polynomials  in $\C[x_0,x_1\ldots, x_n]$  of degree $m$ the set of their common
zeros is typically empty. 
\end{remark}

Proposition \ref{prop:basis} together with our 
numerical experiments motivate the following  question. 

\medskip
 Given an arbitrary infinite matrix $\infmatA$ as above, associate to each 
$\eigenlocus_\infmatA^{(m)}$ its ``root-counting'' measure $\mu_\infmatA^{(m)}$ 
supported on $\eigenlocus_\infmatA^{(m)}\subset \C^{n+1}$ by assigning to every 
point $p\in \eigenlocus_\infmatA^{(m)}$ the point mass ${\kappa(p)}/{\binom
{m+n}{n+1}}$
where $\kappa(p)$ is the multiplicity of $p$. 
(Obviously, $\mu_\infmatA^{(m)}$ is a discrete probability measure.)  

\medskip\noindent
{\bf Main Problem.} \label{prob:1} Under which assumptions on $\infmatA$ does the 
weak limit $\mu_{\infmatA}=\lim_{m\to\infty}\mu_{\infmatA}^{(m)}$ exists? 
In case when $\mu_{\infmatA}$ exists, 
is it possible to describe the support and density of the measure?

\medskip

In the classical case $n=0$, the above problem was intensively studied by many 
authors. The main focus has been when $\infmatA$ is either a Jacobi or a Toeplitz matrix 
(or their generalizations such as block-Toeplitz matrices etc.), see e.g.
\cite{BS, BG, Wi1, Wi2}.

The main goal of this note is to present  a multivariate analogue of the 
well-known  theorem by P.~Schmid and F.~Spitzer \cite{SS},
where they describe $\mu_{\infmatA}$ for an arbitrary banded Toeplitz matrix $\infmatA$  
in the case $n=0.$

Namely, let $\infmatA = (c_{i-j}),$ with $i,j=1,2,\dotsc$ be an infinite, 
banded Toeplitz matrix, where $c_i=0$ if $i<-k$ or $i>h.$
Fixing $n\ge 0$ as above, we obtain for each positive integer $m$ the 
eigenvalue locus $\eigenlocus_\infmatA^{(m)}$ of the 
principal $m \times (m +n)$ submatrix $A^{(m)}$ of $\infmatA$.  

\medskip
Define the \defin{limit set $B_{\infmatA}$ of eigenvalue loci} as
\begin{equation}\label{eqn:Bdef2}
B_\infmatA = \left\{ \xvec \in \C^{n+1} \colon \xvec=\lim_{m \rightarrow \infty}
\xvec_m, \xvec_m \in \eigenlocus_\infmatA^{(m)}
\right\},\; \xvec = (x_0,\dotsc,x_n).
\end{equation}
In other words, $B_\infmatA$ is the set of limit points of the sequence  
$\{\eigenlocus_\infmatA^{(m)}\}$. Thus $B_\infmatA$ is the support of the 
limiting measure $\mu_\infmatA$ if it exists. 
(For a general infinite matrix $\infmatA$ as above, 
its limit set $B_\infmatA$ might be empty.)

Set 
\begin{align}\label{eq:symbol}
Q(t,\xvec) = t^k\left(\sum_{j=-k}^h c_j t^j -\sum_{j=0}^n x_j t^j \right),  
\end{align}
and let 
$\alpha_1(\xvec),\alpha_2(\xvec),\dots,\alpha_{k+h}(\xvec)$ be the  roots of
$Q(t,\xvec)=0,$
ordered according to their absolute values, i.e.  $|\alpha_i(\xvec)|\leq
|\alpha_{i+1}(\xvec)|$ for all $0<i<k+h.$
Let $C_\infmatA $ be the real semi-algebraic set given by the condition: 
\begin{equation}\label{eqn:Cdef2}
C_\infmatA = \{ \xvec \in \C^{n+1} \colon
|\alpha_k(\xvec)|=|\alpha_{k+1}(\xvec)|=\dots=|\alpha_{k+n+1}(\xvec)| \}. 
\end{equation}

Our main conjecture is as follows. 

\begin{conjecture}\label{conj:main} 
For any banded Toeplitz matrix $\infmatA$, 
if $B_\infmatA$ is defined by \eqref{eqn:Bdef2} and $C_\infmatA$ 
defined by  \eqref{eqn:Cdef2} then $B_\infmatA=C_\infmatA$.
\end{conjecture}

By Conjecture~\ref{conj:main} the set $B_\infmatA$ is a 
real semi-algebraic $(n+1)$-dimensional subset of $\C^{n+1}.$ 
 In the classical case $n=0$, 
Conjecture~\ref{conj:main} is settled by P.~Schmidt and F.~Spitzer in
\cite{SS}. 
Another important case when Conjecture~\ref{conj:main} has been proved follows from
some known results on multivariate Chebyshev polynomials, 
which is is presented in  Example~\ref{ex:A} below. 
 Namely, when $k=1$ and $h=n+1$ with $c_{-1}$ and $c_{n+1}$ non-zero,
we may do a affine change of the variables and a scaling of $\infmatA$.
This reduces to the latter case to $c_{-1}=c_{n+1}=1$ and all other $c_i=0.$

For these particular values, the family $\{P_\infmatA^I(\xvec)\}$ 
becomes the multivariate Chebyshev polynomials of the second kind, see e.g.
\cite{DL,KO, BE,Xu}.
These polynomials also have a close connection to another well-known family of
polynomials, namely the Schur polynomials. 

\noindent
\begin{example}\label{ex:A}
{\rm  For the above matrices corresponding 
to the multivariate Chebyshev polynomials the eigenlocus 
$\eigenlocus_\infmatA^{(m)}$ can be described explicitly, 
see for example \cite{GE}. 

More precisely, the points in $\eigenlocus_\infmatA^{(m)}$ lie on
a real, $n$-dimensional surface $C_\infmatA\subset \C^{n+1}$
which is naturally parametrized by  an $(n+1)$-dimensional torus $T^{n+1}$.
This parametrization is given by 
\begin{align}\label{eq:chebylocus}
 C_\infmatA = \left\{ \xvec \in \C^{n+1} | 
x_j =
-e_{j+1}\left(\exp(i\theta_1),\dotsc,\exp(i\theta_{n+1}),\exp(i\theta_{n+2}
)\right)\right\}
\end{align}
where $(\theta_1,\dotsc,\theta_{n+1})$ lie on $T^{n+1},$ $\sum_{j=0}^{n+2}
\theta_j =0,$ and 
$e_j$ is the $j$-th elementary symmetric function in $n+2$ variables.

Notice that for $\xvec \in C_\infmatA,$
\begin{align}
Q(t,\xvec) &= 1+x_0 t + x_1 t^2 + \dotso + x_n t^{n+1} + t^{n+2}= \\
           &= \prod_j (t+ e^{i\theta_{j}})
\end{align}
by the Vieta formulas. Thus, for $\xvec \in C_\infmatA,$ 
\emph{all} roots of $Q$, (as a polynomial in $t$) 
have absolute value equal to  1 when the $x_j$ are parametrized as in
\eqref{eq:chebylocus}. 

Furthermore, the points in $\eigenlocus_\infmatA^{(m)}$ 
are also  expressed by  \eqref{eq:chebylocus}, 
with the parameters $(\theta_1,\dots, \theta_{n+2})$ being certain rational
multiples of $\pi,$
distributed in a regular lattice.
The mapping from the $2$-torus to the eigenlocus is illustrated in
Fig.~\ref{fig1}.

Another interesting aspect of Example~\ref{ex:A} is that all the points
$\xvec=(x_0,\dots, x_n)$  
in the sets $\eigenlocus_\infmatA^{(m)}$
satisfy the conditions $x_j = \overline{x_{n-j}}$, $j=0,1,\dots,n,$
which explains why we can draw $C_\infmatA \subset \C^2$ in Fig.~\ref{fig1}a as 
a 2-dimensional set. 
For larger $n$, $C_\infmatA$ is a   $(n+1)$-dimensional analogue of the 
two-dimensional deltoid, shown in Fig.~\ref{fig1}a.}
\end{example}

For general $\infmatA$, we do not have the inclusion $\eigenlocus_\infmatA^{(m)}
\subseteq C_\infmatA$ for arbitrary finite $m$. 
However, if $\infmatA$ has an additional  extra symmetry, this seems to be case. 

\begin{definition}
A banded Toeplitz matrix such that its $Q(t,\xvec)$ in \eqref{eq:symbol}
satisfies 
\[
\overline{Q(t,x_0,x_1,\dotsc,x_n)} =
\overline{t}^{h+k-1}Q(1/\overline{t},\overline{x}_n,\overline{x}_{n-1},\dotsc,
\overline{x}_0) 
\]
is called \defin{multihermitian} of order $n$.
\end{definition}

\begin{conjecture}\label{conj:conjugates} 
If $\infmatA$ is multihermitian of order $n$, then 
each point $\xvec=(x_0,x_1,\dotsc,x_n) \in \eigenlocus_\infmatA^{(m)}$
satisfies $x_j = \overline{x_{n-j}}$ for $j=0,1,\dotsc,n.$
\end{conjecture}

Conjecture~\ref{conj:conjugates} obviously holds  for the case $n=0,$ 
as it reduces to the fact that hermitian matrices have real eigenvalues.
It is also straightforward to check that Conjecture \ref{conj:conjugates} 
is true for the Chebyshev case of Example~\ref{ex:A} above.

We have extensive numerical evidence for this conjecture.
Another strong indication supporting Conjecture~\ref{conj:conjugates}  is that if $\infmatA$ is multi-hermitian,
then  every point 
$\xvec \in C_\infmatA$ (which by Conjecture \ref{conj:main} is in the limit
eigenlocus) satisfies 
the required symmetry $x_j = \overline{x_{n-j}}$ for $j=0,1,\dotsc,n.$


\begin{figure}[ht!]
\begin{subfigure}[b]{0.48\textwidth}
\centering
  \includegraphics[width=1\textwidth]{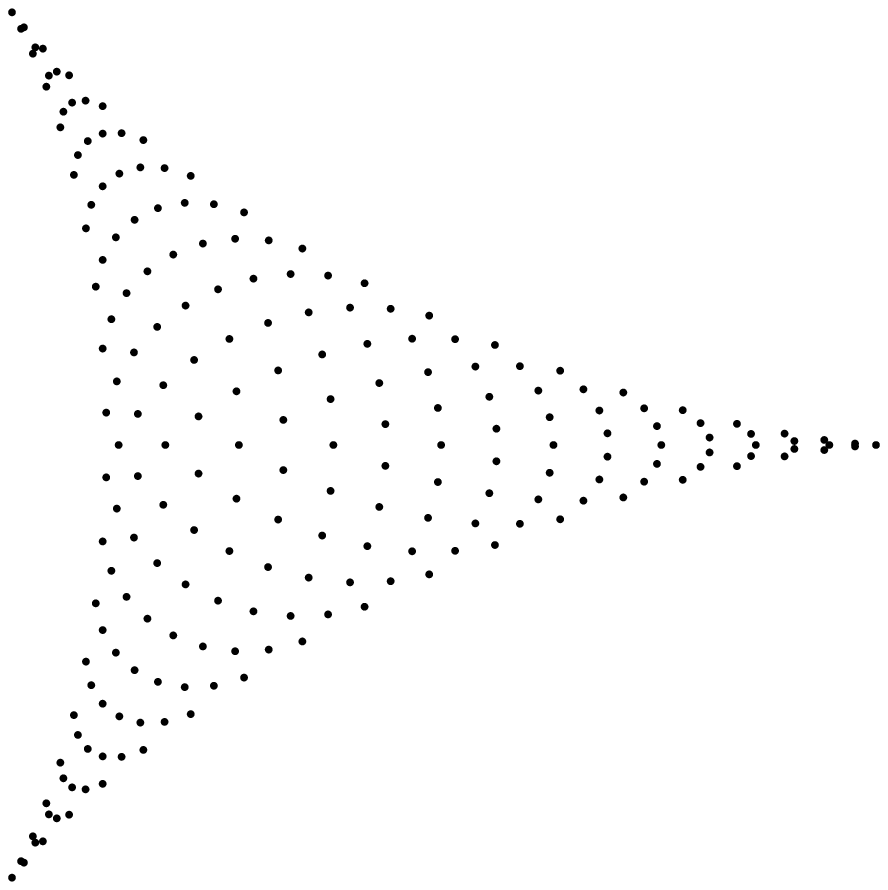}
\end{subfigure}
 \hfill
\begin{subfigure}[b]{0.48\textwidth}
\includegraphics[width=1\textwidth]{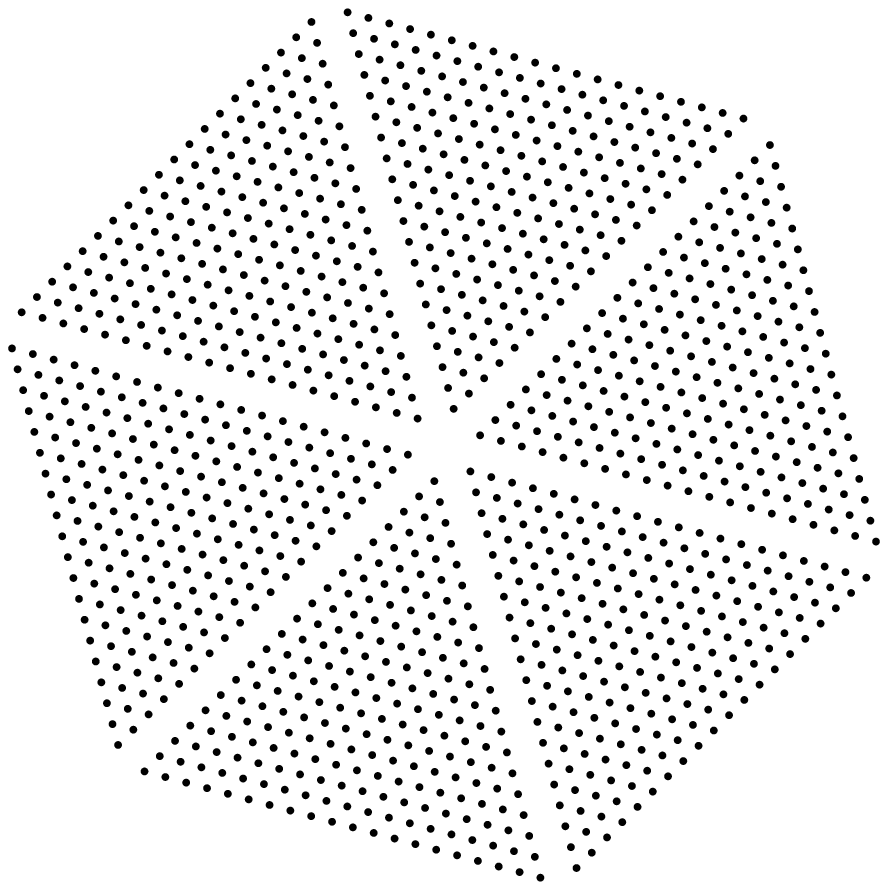}
\end{subfigure}
\caption{
The eigenvalue locus  $\eigenlocus_2^{(20)}$ and its pull-back to $T^2$.
The torus $T^2$ is covered with a hexagon, 
where each triangle is mapped to the eigenlocus. 
The 6-fold symmetry is due to the $S_3$-action by permutation of the 
arguments $\theta_1,\theta_2,\theta_3$ in \eqref{eq:chebylocus}.
(Notice $\theta_1+\theta_2+\theta_3=0$ and this is the subspace which is 
illustrated in the figure to the right.)
}
 \label{fig1}
\end{figure}


The next group of examples are bivariate analogues of 
 special univariate cases originally studied in \cite{SS}, 
and later in \cite{BG}, where they are referred to as ``star-shaped curves'':
\begin{example}\label{ex:B}
{\rm The bivariate case when 
$Q(t,\xvec) = 1 + t^d x_0 + t^{d+1}x_1 + t^{2d+1}$, $d\geq 1$
gives sets in  $\mathbb{C}^2$ where $x_0 = \overline{x_1}$, 
by Conjecture \ref{conj:conjugates}. They correspond to Toeplitz matrices of the
form
\[
\begin{pmatrix}
x_0 & x_1 & 1 & 0 & 0 & \cdots \\
1 & x_0 & x_1 & 1 & 0 & \cdots \\
0 & 1 & x_0 & x_1 & 1 & \cdots \\ 
\vdots & \vdots & \vdots & \vdots & \ddots
\end{pmatrix}, \; 
\begin{pmatrix}
x_0 & x_1 & 0 &1 & 0 & 0 & \cdots \\
0 & x_0 & x_1 & 0 & 1 & 0 & \cdots \\
1 & 0 & x_0 & x_1 & 0 & 1 & \cdots \\ 
\vdots & \vdots & \vdots & \vdots& \vdots & \ddots
\end{pmatrix}, \dotsc
\]
The above two matrices represent  $d=1$ and $d=2.$

Figures~\ref{fig2} and \ref{fig3} present the distributions of $x_0 \in
\mathbb{C},$ for $d=2,3,4$. (Recall that $x_1=\bar{x}_0$.) 
The points shown on these figures belong to  $\eigenlocus_\infmatA^{(m)}$ for
$m=13,14,15$,
and the curves are certain hypocycloids, parametrizing the boundary of
$C_\infmatA$.
More explicitly, for a given integer $d\ge 1$ the hypocycloid boundary for $x_0
\in \mathbb{C}$ is given by
\[
x_0 = (-1)^d e^{-i (d+2) \theta} \left((d+2) e^{i (2 d+3) \theta}+d+1\right)
\text{ where } \theta \in [0,2\pi],
\]
which is one of the implications of Conjecture \ref{conj:main}.
\begin{figure}
\begin{center}
\includegraphics[scale=0.45]{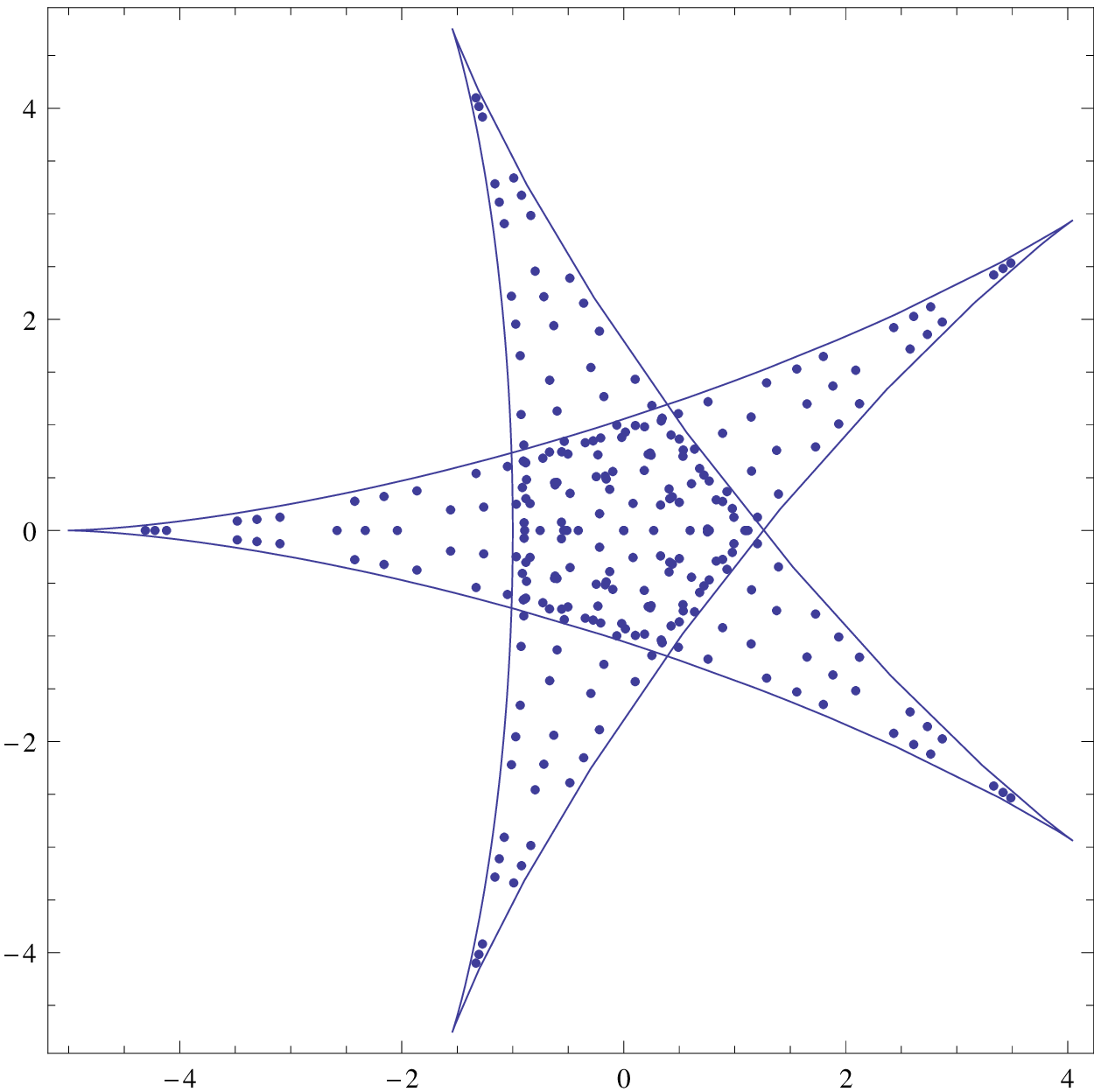}
\includegraphics[scale=0.45]{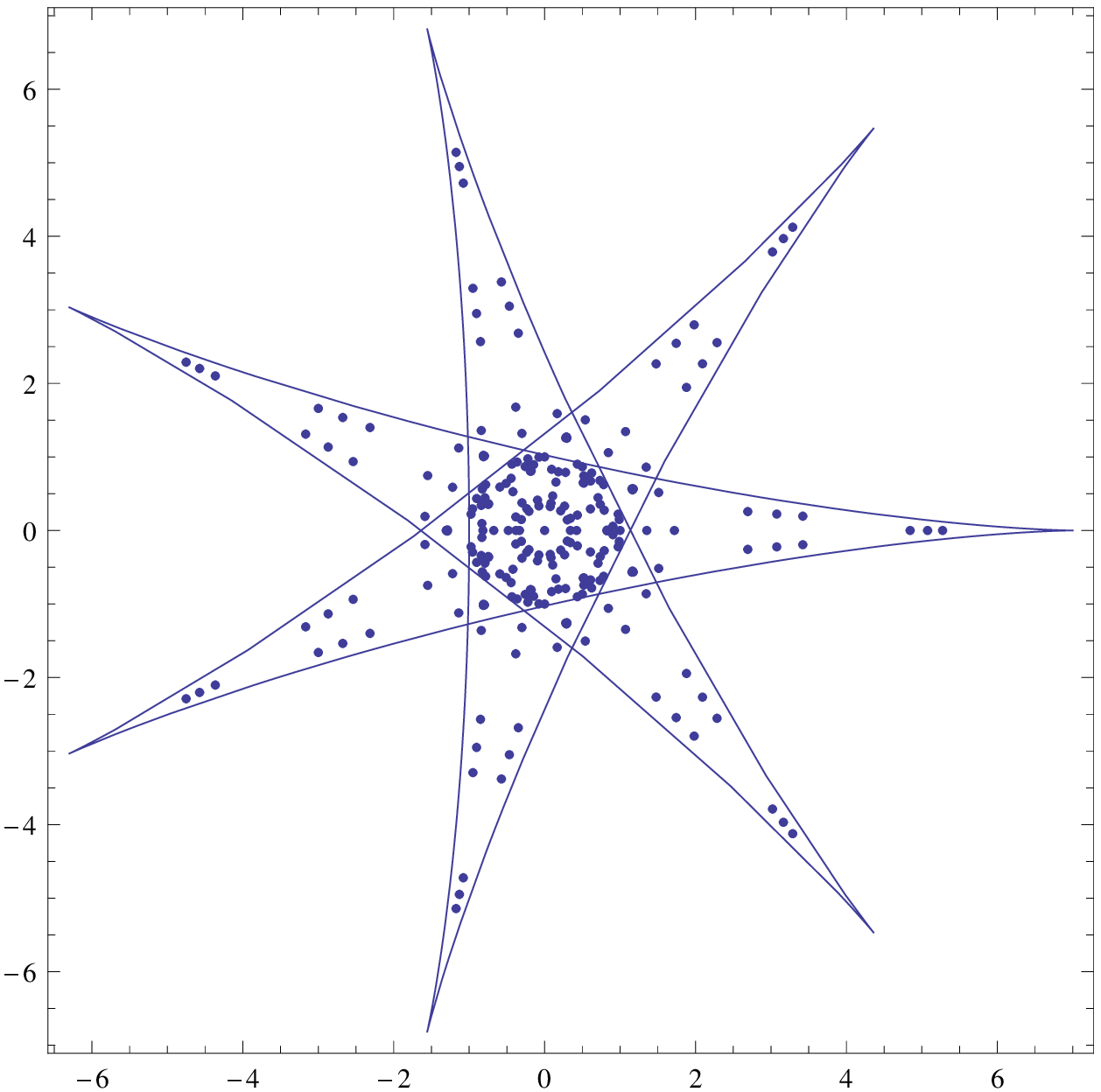}
\end{center}
\vskip 0.3cm
\caption{5-edged star, when $d=2$ and 7-edged star, when $d=3$}
\label{fig2}
\end{figure}

\begin{figure}
\begin{center}
\includegraphics[scale=0.45]{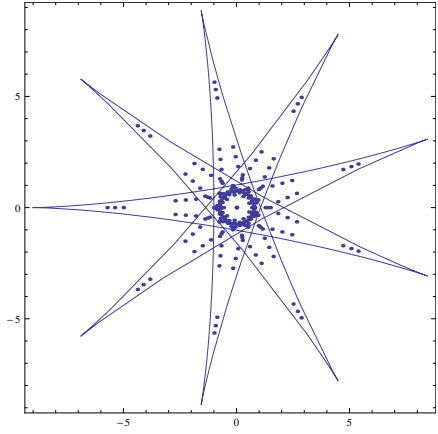}
\end{center}
\vskip 0.3cm
\caption{9-edged star, when $d=4$.}
\label{fig3}
\end{figure}}
\end{example}

\medskip
Finally, the main result of this note is as follows;

\begin{theorem}\label{th:main}
For any banded Toeplitz matrix $\infmatA$, 
where $B_\infmatA$ is defined by \eqref{eqn:Bdef2} and $C_\infmatA$ is defined by 
\eqref{eqn:Cdef2}, 
one has $B_\infmatA\subseteq C_\infmatA$.
\end{theorem}

\medskip\noindent
\textit{Acknowledgements.} The authors are sincerely grateful 
to Professor M.~Tater who actively participated in the consideration of some 
initial examples  related to multivariate Chebyshev polynomials for his help 
and to the Nuclear Physics Institute in  \v Re\v z near 
Prague for the hospitality in March 2011. 
We want to thank Professor A.~Gabrielov of Purdue University 
for his help with the proof of Proposition~\ref{cor:main}. 

\section{Proofs}

\begin{proof}[Proof of Proposition~\ref{prop:basis}] We shall prove items
\rm{(i)} and \rm{(ii)} simultaneously by calculating the leading homogeneous
part of 
$P^I_\infmatA(x_0,...,x_n)$. Let us order the set of all admissible indices
$I=(1\le i_1<\ldots <i_m\le m+n)$ lexicographically. We can also order
lexicographically all monomials of degree $m$ in $x_0,\ldots, x_n$. By
equation~\eqref{eq:main} $P^I_\infmatA(x_0,...,x_n)=\det A_I$ where the columns
of $A_I$ are indexed by $I$. Let $\widetilde P^I_\infmatA(x_0,...,x_n)$ be the
homogeneous part of $P^I_\infmatA(x_0,...,x_n)$ of degree $m$. One can easily
see that the product of all entries on the main diagonal of $A_I$ contains the
 monomial $\mathfrak{m}_I$ of degree $m$ given by  $\mathfrak m_I=\prod_{j=1}^m
x_{i_j-j+1}$. Moreover it is straight-forward that  $\widetilde
P^I_\infmatA(x_0,...,x_n)=\mathfrak{m}_I+\ldots$ where $\ldots$ stands for the
linear combination of monomials $\mathfrak{m}_{I^\prime}$ of degree $m$ coming other
$I^\prime$ which are  lexicographically smaller than $I$. In other words, the
matrix formed by $\widetilde P^I_\infmatA(x_0,...,x_n)$ versus monomials is
triangular in the lexicographic ordering with 
unitary main diagonal which proves items \rm{(i)} and \rm{(ii)}. 

Item \rm{(iii)} is just a reformulation of  Proposition~\ref{prop:eig} above.
\end{proof}

Throughout the rest of the paper, 
we use the convention $\alphavec=(\alpha_1,\dotsc,\alpha_{h+k}).$ 
We will also assume that $c_h=1$, which corresponds to a rescaling of the original matrix $\infmatA.$ 
This is equivalent to the assumption  that $Q(t,\xvec)$ is monic.
By shifting the variables in $\xvec$, we may also assume, without loss of generality, 
that $c_0=c_1\dotso=c_n=0$ in $\infmatA.$

In the above notation, it is convenient to work with the roots of $Q(t,\xvec).$
This motivates the following definitions.
Let $\Gamma_j\subset \C^{h+k},\; j=k,\dots, k+n$ denote the 
real semi-algebraic hypersurface
consisting of all $\alphavec=(\alpha_1,\dotsc, \alpha_{h+k})$ such that
when the $\alpha_j$ are ordered with increasing modulus, $|\alpha_j|=|\alpha_{j+1}|$.
Similarly, let $G_j$ be defined as the real semi-algebraic set 
\[
\{ \xvec \in \C^{n+1} \colon Q(t,\xvec) = (t-\alpha_1)\dotso (t-\alpha_{h+k}) \text{ where } \alphavec \in \Gamma_j \}.
\]
Then, by definition, $C_\infmatA = \bigcap_{j=k}^{k+n} G_j.$

\begin{proposition}\label{prop:compact} 
For any banded Toeplitz matrix $\infmatA$ and any 
non-negative $n<h$, the set $C_\infmatA$ defined by \eqref{eq:symbol}-\eqref{eqn:Cdef2} 
is compact.
\end{proposition}
\begin{proof}
As discussed above, 
we may without loss of generality assume that $c_h=1$ and 
$c_0 = c_1 = \dotso = c_n=0.$ 
Since $Q$ may be assumed to be monic, 
we have $c_j = e_{h-j}(-\alphavec)$ for $-k\leq j<0$ and $n<j\leq h,$
and $x_j = -e_{h-j}(-\alphavec)$ when $0\leq j \leq n.$
Thus, it suffices to show that the set of $\alphavec \in \C^{h+k}$
that satisfies the conditions \eqref{eq:symbol}-\eqref{eqn:Cdef2},  
is compact.
It is also evident that the set $C_\infmatA$ is closed, 
so we only need to show that it is bounded. We show this fact  by contradiction.
\medskip

Assume we have a sequence of roots $\{\alphavec^m\}_{m=1}^\infty$ of \eqref{eq:symbol} 
     such that $\|\alphavec^m\| \rightarrow \infty$ 
where \eqref{eqn:Cdef2} 
is satisfied for each $\alphavec^m.$
We assume that the modulus of the roots are always ordered increasingly.
There are two cases to consider.

\medskip\noindent
\textbf{Case 1:} Assume that for some $0\leq b < k$, a sequence of individual 
roots satisfies the condition $|\alpha_{b+1}^m|\rightarrow \infty$ but 
$|\alpha_j^m|$ are bounded for all $m$ and $j\leq b.$
  Then consider $e_{h+k-b}(\alphavec).$
Since $b < k,$ in our notation $e_{h+k-b}(\alphavec)$ equals the 
coefficient $c_{b-k}.$ 
Notice that $e_{h+k-b}$ contains the term
$\alpha_{b+1} \alpha_{b+2} \dotsm \alpha_{h+k}$ 
which  grows quicker than all
other terms in the expansion of $e_{h+k-b}(\alphavec).$
This contradicts the assumption  $e_{h+k-b}(\alphavec)=c_{b-k}.$

\medskip\noindent 
\textbf{Case 2:} 
Assume that for some $b$ with $k+n\leq b < h+k,$ we have a sequence of 
individual roots $|\alpha_{b+1}^m|\rightarrow \infty$ but $|\alpha_j^m|$ are 
bounded for all $m$ and $j\leq b.$ Consider 
\[
e_{b}(\alphavec)=e_{b}(\alpha_1,\dotsc,\alpha_{h+k}) = \sum_{\sigma \in \binom{[h+k]}{b}}
\frac{e_0}{\alpha_{\sigma_1}\alpha_{\sigma_2} \dotsm \alpha_{\sigma_{b}}}.
\]
This contains an expression with the denominator $\alpha_1 \alpha_2 \dotsm
\alpha_{b}$, i.e. the product of all bounded roots.
Now, since $h+k-b$ roots among all $h+k$ roots grow in absolute value, 
and the product $\alpha_1 \dotsc \alpha_{h+k}$ equals $c_h,$
it follows that $|\alpha_1 \alpha_2 \dotsm \alpha_{b}| \rightarrow 0$ as 
$m\rightarrow \infty,$
and this term converges to 0 quicker than any other product 
$\alpha_{\sigma_1}\alpha_{\sigma_2} \dotsm \alpha_{\sigma_{b}}.$
Thus, $|e_{b}|$ should grow. 
This contradicts the assumption  $e_b(\alphavec) = c_{h-b}.$

Notice that under our assumptions, the above cases  cover all possible 
ways for a sequence of roots to diverge.
Since both cases yield a contradiction, it follows that any sequence of roots of \eqref{eq:symbol}  
satisfying \eqref{eqn:Cdef2} 
must be bounded. Thus, $C_\infmatA$ is compact.
\end{proof}

The following result is multivariate analog of a known fact in the  case $n=0$, see \cite[Prop. 11.18 and 11.19]{BS}.

\begin{proposition}\label{cor:main}
In the  notation of \eqref{eq:symbol}--\eqref{eqn:Cdef2}, 
for any $\xvec$ belonging to  the boundary  $\partial C_\infmatA$ of $C_{\infmatA}$, 
at least one of the following three conditions is  satisfied: 
\begin{enumerate}
\item[(i)] the discriminant of $Q(t,\xvec)$ with respect to $t$ vanishes, i.e. $Q(t,\xvec)$ has (at least) a double root in $t$. 
\item[(ii)]  $|\alpha_{k-1}(\xvec)|=|\alpha_{k}(\xvec)|=|\alpha_{k+1}(\xvec)|=\dots=|\alpha_{k+n+1}(\xvec)|$.
\item[(iii)] $|\alpha_{k}(\xvec)|=|\alpha_{k+1}(\xvec)|=\dots=|\alpha_{k+n+1}(\xvec)|=|\alpha_{k+n+2}(\xvec)|$. 
\end{enumerate} 
\end{proposition}

\begin{proof}
We need the following two simple statements.
\begin{lemma}\label{lm:triv} Let $Pol_d$ be the set of all monic polynomials of degree $d$ with complex coefficients. 
Let $\Sigma_{p,q}\subset Pol_d$ be the subset of polynomials 
satisfying 
\begin{equation}\label{eq:pq}
|\al_p|=|\al_{p+1}|=\cdots =|\al_{q}|, 
\end{equation}
where $1\le p<q\le d$ and $\al_1,\al_2,\dots,\al_d$ being the roots of 
polynomials ordered according to their increasing absolute values. 
Then $\Sigma_{p,q}$ is a real semi-algebraic set of codimension $q-p$ 
whose boundary is the union of three pieces: $\Sigma_{p-1,q}$, $\Sigma_{p, q+1}$ 
and the intersection of $\Sigma_{p,q}$ with the standard discriminant in $Pol_d$, 
i.e. the set of polynomials having multiple roots. 
(Notice that if $p=1$ then $\Sigma_{p-1,q}$ is empty, and if $q=d$ then $\Sigma_{p,q+1}$ is empty by definition.)  
\end{lemma}
\begin{proof}
$\Sigma_{p,q}$ is obtained as the image under the Vieta map of an obvious 
semi-algebraic set $|\al_1|\le|\al_2|\le\dots\le|\al_p|=|\al_{p+1}|=\dots =|\al_q|\le|\al_{q+1}|\le \dots \le |\al_d|$. 
Notice that the Vieta map is a local diffeomorphism outside the preimage of the standard discriminant.  
Therefore the boundary of $\Sigma_{p,q}$ must either belong to the standard discriminant 
or to one of $\Sigma_{p-1,q}$ or $\Sigma_{p,q+1}$. 
The former is the common boundary between $\Sigma_{p,q}$ and $\Sigma_{p-1,q-1}$ and the
latter is the common boundary between $\Sigma_{p,q}$ and $\Sigma_{p+1,q+1}$. 
\end{proof}

Given a closed Whitney stratified set $X$ 
(for example, semi-analytic)  we say that $X$ is \emph{a $k$-dimensional stratified set without boundary} if 
\begin{enumerate}
 \item[(i)] the top-dimensional strata of $X$ have dimension $k$;  
 \item[(ii)] for any point $x$ lying in any stratum  of dimension $k-1$, 
one can choose orientation of the (germs of) $k$-dimensional 
strata of a sufficiently small neighborhood of $x$ in $X$
so that $\partial X=0$.
\end{enumerate}
 
\begin{lemma}\label{lm:inter}
The boundary of the intersection of any closed 
semi-algebraic set $\Gamma$ with any closed algebraic 
set $\Theta$ is included in  the intersection of the boundary $\partial \Gamma$ with $\Theta$. 
\end{lemma}

\begin{proof} 
Observe that any real algebraic variety $X$ of dimension $k$ is a stratifiable set without boundary. 
Indeed, the fact we are proving is local,
and it suffices to prove it for generic $x$ on $(k-1)$-dimensional strata.

Consider an embedding of $X$ in 
some high-dimensional linear space, 
take the Whitney stratification with $x$ on its stratum $Y\subset B$
of dimension $k-1,$ and a transversal to $Y$ of codimension $k-1$ at $x.$ 

Therefore, we may now assume that the germ of $X$ near $x$ 
is topologically a product of a germ of algebraic curve
and a germ of a smooth manifold of dimension $k-1$. 
Furthermore,  a germ of any real algebraic curve $\Gamma$ can be always
oriented so that $\partial\Gamma=0$ which follows from the existence of 
Puiseux series for  an arbitrary branch of algebraic curve.  
This argument shows that any point in the intersection $\Gamma\cap \Theta$ 
which does not belong to the boundary of $\Gamma$ can not lie on the boundary of 
this intersection which settles Lemma~\ref{lm:inter}. 
\end{proof}

Lemmas~\ref{lm:triv} and \ref{lm:inter} immediately imply Proposition~\ref{cor:main} since every $C_{\infmatA}$ is the intersection of an appropriate $\Sigma_{p,q}$ with an appropriate affine subspace in $Pol_{k+h}$. 
\end{proof}

\begin{proof}[Proof of Theorem~\ref{th:main}]
In our notation, let $D^m_j(\xvec)$ be the determinant of the $m \times m$-matrix $A_I$
with $I=\{j+1,j+2,\dotsc,j+m\}$ for $0\leq j \leq n.$ It is evident that
$\eigenlocus_\infmatA^{(m)}$ is a subset of the set
$\widetilde{\eigenlocus}_\infmatA^{(m)}$ of solutions to the system of
polynomial equations 
\begin{align}\label{eq:determinantsystem}
D^m_0(\xvec)=D^m_1(\xvec)=\dotso=D^m_n(\xvec)=0. 
\end{align}
We will show a stronger statement that,  in notation of Theorem~\ref{th:main}, 
\[
\lim_{m\to\infty} \widetilde{\eigenlocus}_\infmatA^{(m)}\subseteq C_\infmatA.
\]  
Although each individual $ \widetilde{\eigenlocus}_\infmatA^{(m)}$ (considered
as a points set with multiplicities) is strictly bigger than
${\eigenlocus}_\infmatA^{(m)}$  the limits $B_\infmatA=\lim_{m\to\infty}
{\eigenlocus}_\infmatA^{(m)}$ and $\lim_{m\to\infty}
\widetilde{\eigenlocus}_\infmatA^{(m)}$ seem to coincide as infinite sets.

The next proposition accomplishes the proof of Theorem \ref{th:main}.
\end{proof}

In Theorem 4 of \cite{Al} it was shown that
each sequence of determinants $\{D^m_j(\xvec)\}_{m=1}^\infty$ as above satisfies a linear
recurrence relation with coefficients depending on $\xvec.$ 
The characteristic polynomial $\chi_j(t)$
of the $j$-th recurrence can be factorized as
\begin{align}\label{eq:chareq}
\chi_j(t,\xvec)=\prod_\sigma (t- r_{j\sigma} ), \text{ where } r_{j\sigma} &= (-1)^{k+j}(\alpha_{\sigma_1}
\cdots \alpha_{\sigma_{k+j}})^{-1}, \sigma \in \binom{[k+h]}{k+j}.
\end{align}

\begin{proposition}\label{pr:main}
Suppose that $\{\xvec_m\}_1^\infty$, is a sequence of
points in $\C^{n+1}$ satisfying the system of equations:
\begin{align}\label{eq:system}
D^m_j(\xvec_m) &=0 \text{ for } j = 0,1,\dotsc,n \text{ and } m=1,2,\dotsc
\end{align}
and such  that the limit $\lim_{m \rightarrow \infty} \xvec_m =:\xvec^*$ exists.
Then for all  $j=0,\dotsc,n$
$ |\alpha_{k+j}(\xvec^*)|=|\alpha_{k+j+1}(\xvec^*)|$
when the $\alpha_i$ are indexed with increasing order of their modulus.
\end{proposition}
\begin{proof}
Provided that all the roots of $\chi_j(t,\xvec)$ are distinct, 
by using a version of Widom's formula, (see \cite{Al,BS}) we have 
\begin{equation}\label{eq:widomform}
D^m_j(\xvec) = \sum_\sigma \prod_{l \in \sigma, i\notin \sigma} \left( 1 - \frac{\alpha_l(\xvec)}{\alpha_i(\xvec)} \right)^{-1}
\cdot r_{j\sigma}(\xvec)^m.
\end{equation}
We may assume that for $\xvec^*$ and fixed $j$, the $r_{j\sigma}(\xvec^*)$ are
ordered decreasingly with respect to their modulus (for some ordering $\sigma=1,2,\dotsc$).
The goal is to prove that $|r_{j1}(\xvec^*)|=|r_{j2}(\xvec^*)|$ 
since this implies $|\alpha_{k+j}(\xvec^*)|=|\alpha_{k+j+1}(\xvec^*)|.$
We show this  fact by contradiction. 
\medskip

Assume that $|r_{j1}(\xvec^*)|>|r_{j2}(\xvec^*)|\geq \dotso \geq |r_{jb}(\xvec^*)|,$ 
i.e. that the largest root is simple and has modulus 
strictly larger than any other root of the characteristic equation \eqref{eq:chareq}.
By examining \eqref{eq:widomform}, 
it is evident that $r_{j1}(\xvec_m)^m$ is the dominating term for sufficiently large $m,$
that is, $D^m_j(\xvec_m)/r_{j1}(\xvec_m)^m \rightarrow L \neq 0$ as $m \rightarrow \infty.$

By standard properties of linear recurrences, this holds
even when there are multiple zeros among the smaller roots;
remember that our assumption was that $r_{j1}(x_m)$ is a simple zero of \eqref{eq:chareq}
when $m$ is large enough.

Hence, for sufficiently large $m$, $D^m_j(\xvec_m) \approx L r_{j1}(x_m)^m,$ 
which is non-zero for sufficiently large $m$. 
This contradicts the condition that $\xvec_m$ satisfies \eqref{eq:system}.
Consequently, $|r_{j1}(\xvec^*)|= |r_{j2}(\xvec^*)|$ for $j=0,1,\dotsc,n$
and this implies Proposition \ref{pr:main}.
\end{proof}

Proposition~\ref{pr:main} implies that $\xvec$ lies in $B_\infmatA$ only if $\xvec$ is a
limit of solutions to \eqref{eq:system},
but such limit $\xvec$ must satisfy that $|\alpha_{k}(\xvec)|=|\alpha_{k+1}(\xvec)|=\dotso =
|\alpha_{k+n+1}(\xvec)|.$ Therefore, $B_\infmatA \subseteq C_\infmatA.$

\section{Further directions}

\noindent
\textbf{1.}  It seems relatively easy to describe the stratified structure 
of $C_\infmatA$ at least in case of generic $\infmatA$. 
In particular, in the Chebyshev case of Example \ref{ex:A} the set $C_\infmatA$ 
has the same stratification as a simplex of corresponding dimension. 
One can also understand the stratified structure of the sets $\Sigma_{p,q}$ 
introduced in Lemma~\ref{lm:triv}. Since each $C_\infmatA$ is obtained from a 
corresponding $\Sigma_{p,q}$ by intersecting it with an affine 
subspace the stratified structure of the former for generic $\infmatA$ is 
also describable. 
On the other hand, our Example~\ref{ex:B} seems to show  more 
complicated stratified structure  due to the presence of additional symmetry.  

\medskip
\noindent
\textbf{2.} We say that an (infinite) complex-valued matrix $\infmatA$  has a {\it 
weak univariate orthogonality property} if the sequence of characteristic
polynomials of its principal minors obeys the standard $3$-term recurrence
relation with complex coefficients. There is a straightforward version of this notion
for finite square matrices.  Obviously, any Jacobi matrix has this property. However,
it seems that for any $m\ge 3$ the set $WO_m\subset Mat(m,m)$ of all $m\times
m$-matrices with the latter property has a bigger dimension than the set $Jac_m
\subset Mat(m,m)$ of all Jacobi $m\times m$-matrices. 

\begin{problem}
Find the dimension of $WO_m$?
\end{problem}

\medskip
\noindent
\textbf{3.} Analogously, given a non-negative integer $n$, we say that  an
(infinite) complex-valued matrix $\infmatA$  has a {\it  weak $n$-variate
orthogonality property} if the above family
$\{P^I_\infmatA(x_0,x_1,\dotsc,x_n)\}$ (see Definition~\ref{deff}) satisfies the
3-term recurrence relation \rm{(2.2)} of Theorem~2.1 of \cite{Xu} with complex
coefficients.

There are many similarities between families
$\{P^I_\infmatA(x_0,x_1,\dotsc,x_n)\}$ and families of multivariate orthogonal
polynomials which by one of the standard  definitions of such polynomials also
satisfy  \rm{(2.2)} of Theorem~2.1 of \cite{Xu} with real coefficients.

Our computer experiments show that in this aspect the case $n>0$ is quite
different from the classical case $n=0$. In particular, we believe that the
following conjecture holds. 

\begin{conjecture} Given $n>0$, a banded matrix $\infmatA$  
has a \textit{weak $n$-variate orthogonality property}
if it is of the form
\[
\infmatA=
\begin{pmatrix}
a_0 & a_1 & a_2 & \dots & a_{n+1} & 0 & 0 & 0 & \dots \\
d_{-1} & d_0 & d_1 &  \dots & d_n & d_{n+1} & 0 & 0 & \dots  \\
0 & d_{-1} & d_0 &  \dots & d_{n-1} & d_{n} & d_{n+1} & 0 & \dots \\
0 & 0 & d_{-1} &  \dots & d_{n-2} & d_{n-1} & d_{n} & d_{n+1} & \dots \\
\vdots & \vdots & \vdots &  \ddots & \vdots & \vdots & \vdots & \vdots & \ddots \\
\end{pmatrix}, 
\]
where $a_0,\dotsc,a_{n+1},d_{-1},\dotsc,d_{n+1} \in \C.$
\end{conjecture}

\end{document}